\documentclass[11pt,reqno,a4paper]{amsart}
\usepackage{amsmath,amssymb,amsthm,amsaddr}
\usepackage{enumitem}
\usepackage[margin=2.5cm,top=3cm,bottom=3cm,footskip=1cm,headsep=1cm]{geometry}
\usepackage[utf8]{inputenc}
\usepackage{graphicx}
\usepackage{tikz,color}
\usepackage{float}

\author{A. Böttcher \and H. Egger}
\address{Department of Mathematics, TU Darmstadt, Germany}
\email{boettcher@mathematik.tu-darmstadt.de,egger@mathematik.tu-darmstadt.de}

\title[Energy stable discretization of Allen-Cahn type models]{Energy stable discretization of Allen-Cahn type problems modeling the motion of phase boundaries}

\newtheorem{lemma}{Lemma}[section]
\newtheorem{problem}[lemma]{Problem}

\theoremstyle{definition}
\newtheorem{remark}[lemma]{Remark}
\newtheorem{notation}[lemma]{Notation}
\newtheorem*{example*}{Example}

\def\div{\mathrm{div}}

\def\dt{\partial_t}
\def\dn{\partial_n}
\def\dx{\partial_x}
\def\dxx{\partial_{xx}}

\def\dtau{\bar\partial_\tau}

\def\RR{\mathbb{R}}

\def\eps{\varepsilon}

\numberwithin{equation}{section}
\numberwithin{table}{section}
\numberwithin{figure}{section}

\begin{document}

\begin{abstract}
We study the systematic numerical approximation of a class of Allen-Cahn type problems modeling the motion of phase interfaces. 
The common feature of these models is an underlying gradient flow structure which gives rise to a decay of an associated energy functional along solution trajectories. We first study the discretization in space by a conforming Galerkin approximation of a variational principle which characterizes smooth solutions of the problem. Well-posedness of the resulting semi-discretization is established and the energy decay along discrete solution trajectories is proven. A problem adapted implicit time-stepping scheme is then proposed and we establish its well-posed and decay of the free energy for the fully discrete scheme. Some details about the numerical realization by finite elements are discussed, in particular the iterative solution of the nonlinear problems arising in every time-step.
The theoretical results are illustrated by numerical tests which also provide further evidence for asymptotic expansions of the interface velocities derived by Alber et al.
\end{abstract}

\maketitle

\vspace*{-1em}

\begin{quote}
\end{quote}

\begin{quote}
\noindent
{\small {\bf AMS-classification (2000):} 65M60,74N20,35J93,53C44}
\end{quote}

\section{Introduction} \label{sec:intro}

We consider the systematic numerical approximation of a class of Allen-Cahn type equations describing, for instance, the motion of anti-phase boundaries in crystalline solids \cite{AllenCahn79}, the evolution of interfaces in more general phase field models \cite{BloweyElliott93,FriedGurtin94,Visintin96}, or the geometric evolution by mean curvature \cite{DeckelnickDziukElliot05,Garcke00}.
For motivation of our considerations, let us briefly consider two particular examples which have been studied in detail in \cite{Alber15,AlberZhu13}. 
As a first model problem, we consider the generalized Allen-Cahn equation
\begin{align} \label{eq:ac}
\dt S^{ac} =  \left( \mu^{1/2}  \lambda \Delta_x S^{ac} - \frac{1}{\mu^{1/2}} \psi'(S^{ac}) - F'(S^{ac}) \right) \frac{c}{(\mu \lambda)^{1/2}}.
\end{align}
Here $S^{ac}$ is an order parameter which takes values close to zero in one phase of the medium and close to one in the other, $\psi(S)=4 S^2(1-S)^2$ is the double well potential, $F'(S)$ is a source term due to external loading, and $c$, $\lambda$ are model parameters. The coefficient $\mu$ plays the role of a regularization parameter that allows to widen the transition zone between the phases and thus to alleviate the numerical simulation. 
For the choice $c=1$, $\lambda = 1$, $\mu=\eps^2$, and $F'(S) = 0$, one obtains the standard Allen-Cahn equation
\begin{align} \label{eq:ac0}
\dt S = \Delta S - \frac{1}{\eps^2} \psi'(S)
\end{align}
which has been studied intensively in the literature; see e.g. \cite{EvansSonerSouganidis92,Ilmanen93,RubinsteinSternbergKeller89} and the references given there. 
It is well-known \cite{Garcke13} that \eqref{eq:ac0}, and therefore also \eqref{eq:ac}, 
can be understood as a gradient flow for an associated free energy functional, which for the problem \eqref{eq:ac} has the form
\begin{align} \label{eq:eac}
E^{ac}(S) = \int_\Omega \frac{\mu^{1/2} \lambda}{2} |\nabla_x S|^2 + \frac{1}{\mu^{1/2}}  \psi(S) + F(S) \;  dx.
\end{align}
The multiplicative term $\frac{c}{(\mu \lambda)^{1/2}}$ on the right hand side of \eqref{eq:ac} yields a scaling of time that leads to a non-trivial behavior for $\mu \to 0$, which is called the {\em sharp interface limit}.
It is shown in \cite{Alber15,AlberZhu13} that the velocity of the interface $\Gamma(t)=\{x : S^{ac}(x)=1/2\}$ can be expanded as 
\begin{align} \label{eq:vac}
v^{ac}(\mu;\lambda) =  \frac{c}{c_1} n_\Gamma \cdot [C] \cdot n_\Gamma + c \lambda^{1/2} \kappa_{\Gamma} + O(\mu^{1/2}),
\end{align}
where $c_1$ can be expressed expicitly in terms of the problem data.
Here $[C]$ denotes the jump of the Eshelby tensor \cite{Eshelby59,Eshelby61} which is related to the external forces $F'(S)$
and $n_\Gamma$ is the normal vector while $\kappa_{\Gamma}$ is the mean curvature of the interface $\Gamma$. 
In the limit $\mu \to 0$, the velocity of the interface motion can thus be described by the well-known relation \cite{Ilmanen93,RubinsteinSternbergKeller89}
\begin{align} \label{eq:vsi}
\bar v(\lambda) = \frac{c}{c_1} n_\Gamma \cdot [C] \cdot n_\Gamma + c \lambda^{1/2} \kappa_{\Gamma}.
\end{align}
Choosing the regularization parameter $\mu>0$ leads to an incorrect speed of the interface motion, 
and according to \eqref{eq:vac}, one should choose $\mu = o(\lambda)$ to keep the consistency error small.
For parameter $\mu>0$ the interface is widened to a transition zone of width $O((\mu\lambda)^{1/2})$.
In order to resolve the smoothed interface in simulations,
the mesh size should therefore be chosen at least as small as $h = O((\mu \lambda)^{1/2}) = o(\lambda)$.
More restrictive assumptions on the mesh size are required to really 
prove a good approximation of the interface evolution; see \cite{FengProhl03,NochettoVerdi97} for details.
If $\lambda$ is small, then a very small mesh size is thus required to obtain a good approximation.

In order to allow for a larger mesh size and to reduce the computational burden, 
the following model for the interface motion has been proposed and analyzed in \cite{Alber15,AlberZhu13}
\begin{align} \label{eq:hyb}
\dt S^{hyb} = \left(\nu \Delta_x S^{hyb} - \psi'(S^{hyb})  - F'(S^{hyb}) \right) |\nabla S^{hyb}|.
\end{align}
This equation, which has been termed \emph{hybrid model} in \cite{Alber15,AlberZhu13}, 
describes the motion of the levelsets of the order parameter $S^{hyb}$ with normal speed proportional to the term in parenthesis. 
Following the arguments of \cite{ClarenzEtAl05,DroskeRumpf04,OsherParagios}, the problem \eqref{eq:hyb}
can again be interpreted as a gradient flow for an associated energy functional, which here reads
\begin{align} \label{eq:ehyb}
E^{hyb}(S) = \int_\Omega \frac{\nu}{2} |\nabla_x S|^2 + \psi(S) +  F(S) \; dx.
\end{align}
The gradient here has to be defined with respect to a metric induced by the solution $S$ itselve; 
we refer to \cite{ClarenzEtAl05,DroskeRumpf04} and to Section~\ref{sec:variational} for details.
Using asymptotic analysis, the propagation speed of the interface has been shown in \cite{Alber15,AlberZhu13} to behave like
\begin{align} \label{eq:vhyb}
v^{hyb}(\nu) = n \cdot [C] \cdot n + \nu^{1/2} \omega_1 \kappa_\Gamma + o(\nu^{1/2}),
\end{align}
where the parameter $\omega_1$ can again be computed from the model data. 

A comparison of the two formulas \eqref{eq:vac} and \eqref{eq:vhyb} reveals that for a proper choice of the model and scaling parameters, the hybrid model leads to the same propagation speed for the interface $\Gamma$ as in the Allen-Cahn model up to higher order terms. 
To obtain a good agreement of the results, one should choose $\nu \approx \lambda$ and in this case width of the smoothed interface in the hybrid model \eqref{eq:hyb} can be shown to behave like $O(\nu^{1/2})=O(\lambda^{1/2})$; see \cite{Alber15,AlberZhu13} for details.
When $\lambda$ is small, the smoothed interface of the hybrid model is thus substantially larger then that of the corresponding Allen-Cahn model and a mesh size of $h=O(\lambda^{1/2})$ should be sufficient to resolve the interface in computations, instead of $h=o(\lambda)$ required for the corresponding Allen-Cahn model. 
The hybrid model should therefore allow to compute the interface evolution on a much coarser mesh, which was the main motivation for the proposal of this method.

Due to the close connection to mean curvature flow, there has been much interest in the numerical approximation for the Allen-Cahn equation \eqref{eq:ac0}. 
Various discretization schemes have been proposed and analyzed, see e.g.  \cite{ChenElliottGardinerZhao98,ShenYang10} and the references given there.
Particular aspects of the efficient numerical approximation have been addressed in \cite{FengSongTangYang13,FengTangYang13,ShenTangYang16}
and a-posteriori error estimates have been derived in \cite{Bartels05,FengWu05,KesslerNochettoSchmidt04}.
The convergence of numerical approximations of the Allen-Cahn equation in the sharp interface limit $\eps \to 0$ towards the generalized motion by mean curvature has been established in \cite{FengProhl03,NochettoVerdi97}; see also \cite{ZhangDu09} for further numerical tests.
The direct numerical approximation of mean curvature flow has also been investigated intensively; see for instance \cite{DeckelnickDziuk95,Walkington96} and refer to \cite{DeckelnickDziukElliot05,Garcke13} for a review about available results and further references.

\medskip

The focus of the current manuscript is somewhat different to these previous works: 
Here we want to study the systematic numerical approximation of a large class of Allen-Cahn type models 
which includes \eqref{eq:ac}, \eqref{eq:ac0}, and \eqref{eq:hyb} as special cases. 
The common feature of all these models is the underlying gradient flow structure with respect to an associated energy functional, and we aim to strictly preserve this structure in the approximation process on the semi-discrete and the fully discrete level.  
With this in mind, we automatically obtain discretization schemes that are uniformly energy stable 
and thus consistent with the second law of thermodynamics.

In this paper we establish well-posedness and qualitative properties of a discretization strategy appropriate for a large class of Allen-Cahn type problems. The models \eqref{eq:ac} and  \eqref{eq:hyb} introduced above will be considered as particular examples. As a by-product of our investigations, we also obtain further numerical evidence for the asymptotic expansions \eqref{eq:vac} and \eqref{eq:vhyb} of the interface velocities.
Our main arguments can be extended to elastic Allen-Cahn type models which describe the evolution of phase interfaces in binary alloys \cite{AlberZhu13,FriedGurtin94,Garcke00}. The required coupling to the additional elasticity system will be investigated in a forth-comming publication.

\medskip

The remainder of the manuscript is organized as follows:
In Section~\ref{sec:variational}, we first introduce a unified formulation of the problems discussed in the introduction and then present a variational characterization of smooth solutions. 
In addition, we verify that the respective energies decrease along solution trajectories which 
is due to the underlying gradient flow structure. 
In Section~\ref{sec:galerkin}, we then discuss the Galerkin approximation of the underlying variational principles. We establish the well-posedness of the resulting semi-discrete schemes 
and verify the energy decay on the discrete level. 
In Section~\ref{sec:time}, we then discuss the subsequent discretization in time by a problem adapted implicit time stepping scheme. We establish the well-posedness of the fully discrete problem and again prove the decay in energy. 
A particular space discretization by finite elements is discussed in Seciton~\ref{sec:implementation},
and some further aspects of the implementation are presented.
For illustration of our theoretical results, we report in Section~\ref{sec:numerics} about some numerical tests,in which we also verify the asymptotic expansions for the interface velocities given in \eqref{eq:vac} and \eqref{eq:vhyb}. 
We close with a short discussion of our results and open problems for future research.

\section{A unified variational framework} \label{sec:variational}

For the rest of the manuscript, we consider the following general model problem:
\begin{align}
c(|\nabla S|) \dt S &= \div_x (\alpha \nabla_x S) - \beta d'(S), \qquad x \in \Omega, \ t>0.  \label{eq:sys1}
\end{align}
Here $\Omega \subset \RR^d$, $1 \le d \le 3$ is a bounded Lipschitz domain, $\alpha,\beta>0$ are model parameters, and the functions $c(\cdot)$, $d(\cdot)$ are given a-priori. For our analysis later on, we will assume that
\begin{itemize}\itemsep1ex
 \item[(A1)] $c \in W^{1,\infty}(\RR)$ and $0 < \underline c \le c(x) \le \overline c < \infty$;
 \item[(A2)] $d \in W^{2,\infty}(\RR)$ with $d(x) \ge 0$, $d(x) \ge \underline d |x| -d_1$, and $|d''(x)| \le \overline d$ for all $x \in \RR$.
\end{itemize}
Some of these technical assumptions are made only for convenience and could further be relaxed
without problems. This will become clear from our analysis below and we will make some remarks in this direction.
To complete the model description, we additionally assume that
\begin{align}
               \alpha \dn S &= 0, \qquad x \in \partial\Omega, \ t>0.  \label{eq:sys2}
\end{align}
Other boundary conditions can also be treated without much difficulty.
The above problem is of quasilinear parabolic type and existence of solutions for appropriate initial values 
$S(0)=S_0$ can be proven with standard arguments; see e.g. \cite{Amann,Ladyzhenskaya,Pazy}.
\begin{notation}
A function $S : C^{1,2}([0,T] \times \overline \Omega)$ is called smooth solution of \eqref{eq:sys1}--\eqref{eq:sys2}, if it satisfies the respective equations in a pointwise sense. 
\end{notation}
Here $C^{1,2}([0,T] \times \overline \Omega)$ denotes the space of continuous functions that are continuously differentiable with respect to $t$ and twice continuously differentiable with respect to $x$. 
Note that acording to our definition, smooth solutions and their spatial derivatives are uniformly bounded. 
\begin{remark}
The Allen-Cahn model \eqref{eq:ac} can be phrased in the form \eqref{eq:sys1} 
by setting  $\alpha=\mu^{1/2} \lambda$, $\beta=\mu^{-1/2}$, $c(x)=(\mu\lambda)^{1/2}/c$, 
and defining the potential as $d(S) = \psi(S) + \mu^{1/2} F(S)$. 
Note that the second derivative $d''(\cdot)$ of the potential is unbounded here. Using maximum principles, the solution $S$ can however be shown to remain bounded and therefore $d(S)$ could be modified for large $S$ in order to satisfy (A2) without changing the solution;
see \cite{FengProhl03} for other appropriate conditions on the potential. 
The hybrid model \eqref{eq:hyb} can also be cast into the form \eqref{eq:sys1} with $\alpha=\nu$, $\beta=1$, $d(s)=\psi(s) +F(S)$, and $c(x) = 1/x$. To avoid problems for vanishing gradient, one can set
% $c(x) = (\eps + 1/(\eps + x^2))^{1/2}$ 
$c(x) = 1/\max(\delta,\min(1/\delta,x))$ for some $\delta>0$ instead, and thus obtain a regularized version of \eqref{eq:hyb}. 
% We will comment in more detail on the appropriate choice of the regularization parameter $\eps$ below. 
%
The limit $\delta \to 0$ may be studied in the framework of $\Gamma$-convergence \cite{DeGiorgi80,Jian99}.
\end{remark}

As a next step, let us define an associated energy functional for the problem \eqref{eq:sys1}--\eqref{eq:sys2} by
\begin{align} \label{eq:energy}
E(S) = \int_\Omega \frac{\alpha}{2} |\nabla S|^2 + \beta d(S) dx.  
\end{align}
Equation \eqref{eq:sys1} can then be interpreted as a gradient flow for this energy with respect to the metric
$g_S(v,w) = \int_\Omega v w\;  c(|\nabla S|) dx$. 
Problem \eqref{eq:sys1} can thus be formally expressed as 
\begin{align} \label{eq:gradflow}
\dt S = - \text{grad}_{g_S} E(S),
\end{align}
which, following \cite{ClarenzEtAl05,DroskeRumpf04,Garcke13}, has to be understood in the sense that
\begin{align} \label{eq:gradflow_var}
\int_\Omega c(|\nabla S|) \dt S v \; dx 
= :g_S(\dt S,v) 
= - \langle \delta E(S), v\rangle
:= -\int_\Omega \alpha \nabla S \nabla v + \beta d'(S) v \; dx 
\end{align}
for all $v \in C_0^\infty(\Omega)$ and $t>0$. The last term is the negative directional derivative of the energy functional in direction $v$.
These formal arguments motivate the following variational principle.
\begin{lemma}[Variational characterization] \label{lem:variational}
Let $S$ be a smooth solution of \eqref{eq:sys1}--\eqref{eq:sys2}. Then 
\begin{align} \label{eq:variational}
\left(c(|\nabla S(t)|) \dt S(t),v\right) + (\alpha \nabla S(t), \nabla v) + (\beta d'(S(t)),v) =0
\end{align}
for all $v \in H^1(\Omega)$ and all $t > 0$. Here $(v,w)=\int_\Omega v w \; dx$ is the scalar product of $L^2(\Omega)$.
\end{lemma}
As a direct consequence of the gradient flow structure, we obtain 
\begin{lemma}[Energy decay] \label{lem:decay}
Let $S$ be a smooth solution of \eqref{eq:sys1}--\eqref{eq:sys2}. Then 
\begin{align} \label{eq:decay}
\frac{d}{dt} E(S(t)) = -\int_\Omega c(|\nabla S(t)|) |\dt S(t)|^2 dx \le 0.
\end{align}
\end{lemma}
\begin{proof}
Since $S$ is smooth, we can apply the chain rule to get 
\begin{align*}
\frac{d}{dt} E(S(t)) 
&= \int_\Omega \alpha \nabla S(t) \nabla (\dt S(t)) + \beta d'(S(t)) \dt S(t) dx = (*).  
\end{align*}
By Lemma~\ref{lem:variational}, any smooth solution of \eqref{eq:sys1}--\eqref{eq:sys2} also satisfies the variational principle \eqref{eq:variational}, and testing this variational equation with $v = \dt S(t)$ immediately yields 
\begin{align*}
(*) 
&= (\alpha \nabla S(t), \nabla \dt S(t)) + (\beta d'(S(t)), \dt S(t)) \\
&= -( c(|\nabla S(t)|) \dt S(t), \dt S(t)) = -\int_\Omega c(|\nabla S(t)|) |\dt S(t)|^2 dx \le 0.
\end{align*}
For the last step we only had to use that $c(\cdot)$ is non-negative.
\end{proof}

\section{Galerkin approximation in space} \label{sec:galerkin}

Motivated by the results of the previous section, we now consider the Galerkin approximation of the variational principle \eqref{eq:variational} to construct semi-discretizations in space. 
Let us choose some finite dimensional subspace $V_h \subset H^1(\Omega)$ and consider the following discrete variational problem.
\begin{problem}[Galerkin semi-discretization] \label{prob:semi} $ $\\
Let $S_{h,0} \in V_h$ and $T>0$ be given. Find $S_h \in C^1([0,T];V_h)$ such that $S_h(0) = S_{h,0}$ and
\begin{align} \label{eq:variationalh}
\Big(c(|\nabla S_h(t)|) \dt S_h(t),v_h\Big) + (\alpha \nabla S_h(t), \nabla v_h) + (\beta d'(S_h(t)),v_h) = 0 
\end{align}
for all test functions $v_h \in V_h$ and for all $ 0 \le t \le T$. 
\end{problem}
Since we employ a conforming Galerkin approximation, the gradient flow structure and thus also the energy decay estimates 
are inherited automatically by the semi-discrete problem.
\begin{lemma} \label{lem:decayh}
Let $S_h$ denote a solution of Problem~\ref{prob:semi}. Then 
\begin{align} \label{eq:decayh}
\frac{d}{dt} E(S_h(t)) = -\int_\Omega c(|\nabla S_h(t)|) |\dt S_h(t)|^2 dx \le 0.
\end{align}
\end{lemma}
\begin{proof}
The assertion follows with literally the same arguments as in Lemma~\ref{lem:decay}. 
\end{proof}

As a next step, let us establish the well-posedness of the semi-discrete problem. 
\begin{lemma}[Well-posedness]
For any $S_{h,0} \in V_h$ and $T>0$, Problem~\ref{prob:semi} has a unique solution. 
\end{lemma}
\begin{proof}
After choosing a basis for $V_h$, we can rewrite the discrete problem as initial value problem
\begin{align*}
M(y) y'(t) + A y + B(y) = 0, \qquad y(0)=y_0,
\end{align*}
where $y(t)$ here denotes the coordinate vector when expanding $S_h(t)$ in the chosen basis. 
Due to assumptions (A1)--(A2) on the coefficients and since $V_h$ is finite dimensional, 
the matrix functions $M(\cdot)$ and $B(\cdot)$
can be seen to be Lipschitz continuous. From the lower bounds on $c(\cdot)$ we further obtain that $M(y)$ is symmetric positive definite.  Existence of a unique local solution then follows from the Picard-Lindelöf theorem. 
From \eqref{eq:energy} and \eqref{eq:decayh}, we can see that
\begin{align*}
C' := E(S_{0,h}) \ge E(S_h(t)) = \frac{\alpha}{2} \|\nabla S_h(t)\|_{L^2(\Omega)}^2 + \beta \|d(S_h(t))\|_{L^1(\Omega)}.
\end{align*}
Here we used that $d(\cdot) \ge 0$ due to assumption (A2). 
This yields $\alpha \|\nabla S_h(t)\|_{L^2\Omega)}^2 \le 2 C'$ 
and by the growth condition and the positivity assumption for $d$, we further obtain
\begin{align*}
C' \ge \beta \|d(S_h(t))\|_{L^1(\Omega)} \ge \beta \underline d \|S_h(t)\|_{L^1(\Omega)}.
\end{align*}
By the Poincar\'e inequality and positivity of $\alpha,\beta$, we thus get $\|S_h(t)\|_{H^1(\Omega)} \le C''$
with constant $C''$ only depending on the domain, on the bounds in the assumptions, and on the energy of the initial value.
This implies that $y(t)$ remains uniformly bounded and hence the solution can be extended uniquely to arbitrarily large $t>0$.
\end{proof}

\begin{remark}
Under assumptions (A1)--(A2), we thus obtain global existence of a unique semi-discrete solution $S_h$.
The constant $C''$ in the a-priori bound $\|S_h(t)\|_{H^1(\Omega)} \le C''$ only depends on the initial energy $E(S_{h,0})$, on the values of $\alpha$ and $\beta$, and on the lower bound $\underline d$ for the growth rate of the potential $d(\cdot)$. 
Let us note that the energy decay estimate also provides a uniform bound for the time derivative of the solution which might be useful for a full convergence analysis for the Galerkin approximation; we refer to \cite{ChenElliottGardinerZhao98,FengProhl03} for results in this direction.
\end{remark}

\section{Time discretization} \label{sec:time}

As a second step in the approximation process, 
let us now discuss the discretization in time. 
Given some time step $\tau>0$, we define $t^n = n \tau$ for $n \ge 0$, and we denote by
\begin{align*}
\dtau S^n = \frac{1}{\tau} (S^{n} - S^{n-1}) 
\end{align*}
the backward difference quotient at time $t^n$. As will become clear from our analysis, non-uniform time steps could be considered as well.
For the time discretization of the Galerkin approximation defined in Problem~\ref{prob:semi}, 
we then consider the following implicit time-stepping scheme.
\begin{problem}[Fully discrete scheme] \label{prob:full} $ $\\
Set $S_h^0 = S_{h,0}$ with initial value as in Problem~\ref{prob:semi}. 
For any $n \ge 1$, find $S_h^n \in V_h$ such that 
\begin{align} \label{eq:variationalhh}
\Big(c(|\nabla S_h^{n-1}|) \dtau S_h^n,v_h\Big) + (\alpha \nabla S_h^n, \nabla v_h) + \Big(\beta \frac{d(S_h^n)-d(S_h^{n-1})}{S_h^n-S_h^{n-1}},v_h\Big) = 0 
\end{align}
for all $v_h \in V_h$. We set $\frac{d(S_h^n)-d(S_h^{n-1})}{S_h^n-S_h^{n-1}} = d'(S_h^n)$ wherever $S_h^n-S_h^{n-1}=0$.
\end{problem}
Let us note that a similar treatment of the nonlinear term has been proposed in the context of mean-curvature flow in \cite{Walkington96}. 
The particular of the discrete variational problem \eqref{eq:variationalhh} allows us to establish a decay estimate for the energy of the fully discrete approximation with very similar arguments as used for the analysis on the continuous level. 
\begin{lemma} \label{lem:decayhh}
Let $\{S_h^n\}_{n \ge 0} \subset V_h$ denote a solution of Problem~\ref{prob:full}.
Then 
\begin{align*}
\dtau E(S_h^n) 
&= \frac{1}{\tau} \Big(E(S_h^n) - E(S_h^{n-1})\Big)  \\
&= -\int_\Omega c(|\nabla S_h^{n-1}|) |\dtau S_h^n|^2 dx - \frac{\alpha}{2\tau} \|\nabla S_h^n - \nabla S_h^{n-1}\|^2_{L^2(\Omega)} \le 0.
\end{align*}
\end{lemma}
\begin{proof}
Using definition \eqref{eq:energy}, we can decompose the discrete energy as
\begin{align*}
&E(S_h^n) - E(S_h^{n-1}) 
= \frac{\alpha }{2} \Big( (\nabla S_h^n,\nabla S_h^n) - (\nabla S_h^{n-1},\nabla S_h^{n-1})\Big) + \beta\Big(d(S_h^n) - d(S_h^{n-1}),1\Big)\\
&= \left(\alpha \nabla S_h^n,\nabla S_h^n - \nabla S_h^{n-1}\right) 
  - \frac{\alpha}{2} \|\nabla S_h^n - \nabla S_h^{n-1}\|^2 + \left(\beta \frac{d(S_h^n)-d(S_h^{n-1})}{S_h^n-S_h^{n-1}},S_h^n-S_h^{n-1}\right). 
\end{align*}
We now utilize the fully discrete variational principle with $v_h = S_h^n - S_h^{n-1}$ to replace the first and last term together by $-\tau (c(|\nabla S_h^{n-1}|) \dtau S_h^n, \dtau S_h^n)$ which directly yields the result. 
\end{proof}

The second term on the right hand side of the above energy estimate stems from numerical dissipation of the implicit time-stepping scheme. The particular time discretization thus enhances the stability of the numerical solution.
As further theoretical backup, we next establish the well-posedness of the fully discrete scheme.
\begin{lemma}
Let $S_h^{n-1} \in V_h$ be given. Then \eqref{eq:variationalhh} has at least one solution $S_h^n \in V_h$, 
and the solution is unique for all time steps $0<\tau \le \tau_0$ with some $\tau_0$ sufficiently small.
\end{lemma}
\begin{proof}
Existence of a solution $S_h^n$ follows from the energy estimate given in the previous section and the Brouwer fixed point theorem, and we know that $\|S^n\|_{H^1(\Omega)} \le C''$.
Now let $S_h^n$ and $\tilde S_h^n$ be two solutions of \eqref{eq:variationalhh}. 
Then testing with $v_h=S_h^n - \tilde S_h^n$ and
using assumption (A1) leads to
\begin{align*}
\frac{\underline c}{\tau} \|S_h^n - \tilde S_h^n\|_{L^2(\Omega)}^2 &+ \alpha  \|\nabla S_h^n - \nabla \tilde S_h^n\|_{L^2(\Omega)}^2 \\
&\le \beta \|D(S_h^n,S_h^{n-1}) - D(\tilde S_h^n,S_h^{n-1})\|_{L^2(\Omega)} \|S_h^n - \tilde S_h^n\|_{L^2(\Omega)}.
\end{align*}
For ease of notation, we utilized the symbol 
\begin{align} \label{eq:DS}
D(S,S^{n-1}) = \frac{d(S) - d(S^{n-1})}{S-S^{n-1}}
\end{align} 
here to abbreviate the difference quotient.
By the fundamental theorem of calculus and using the bounds of assumption (A2) for the coefficient $d(\cdot)$, 
we further get
\begin{align*}
|D(S^n,S^{n-1}) - D(\tilde S^n,S^{n-1})| 
&\le \int_0^1 \big| d'(S^{n-1} + \xi (S^n - S^{n-1})) - d'(S^{n-1} + \xi (\tilde S^n  - S^{n-1})) \big|\; d\xi \\
&\le \max \{|d''(S)| : S \in \RR\} |S^n - \tilde S^n| \le \overline d |S^n - \tilde S^n|.
\end{align*}
A combination of the two estimates then directly leads to 
\begin{align*}
\frac{\underline c}{\tau} \|S_h^n - \tilde S_h^n\|_{L^2(\Omega)}^2 + \alpha  \|\nabla S_h^n - \nabla \tilde S_h^n\|_{L^2(\Omega)}^2
\le \overline d \|S_h^n - \tilde S_h^n\|_{L^2(\Omega)}^2.
\end{align*}
For $\tau>0$ small enough, the right hand side can be absorbed by the terms on the left side 
and we thus also obtain the uniqueness of the solution.
\end{proof}

Before we proceed, let us briefly indicate at this point how the assumption (A1)--(A2) could be possibly relaxed,
which can be deduced by inspection of the previous proof.
\begin{remark} \label{rem:ass}
The bounds for $|d''(S)|$ are not required for all $S \in \RR$ here, but only for values $S$ that are actually attained during the evolution. 
If a bound $|S_h^n| \le C_h''$ is available, which may be proved by a discrete comparison principles \cite{ShenTangYang16}, 
then one could replace the constant $\overline d$ in the above proof by $\overline d' = \sup \{ |d''(S)| : |S| \le C_h''\}$.
If the space $V_h$ admits an inverse inequality 
\begin{align*} 
\|v_h\|_{L^\infty(\Omega)} \le C_h \|v_h\|_{L^2(\Omega)} \qquad \text{for all } v_h \in V_h, 
\end{align*}
then the a-priori estimate $\|S_h^n\|_{H^1(\Omega)} \le C^{''}$ implies $|S_h^n| \le C_h^{''}$ with $C_h'' = C_h C^{''}$.
The constant $C_h$ can be estimated explicitly for many discretization schemes, e.g., for finite element methods which discussed in the following section.
In one space dimension, the above inverse inequality holds true for any choice of $V_h$ due to the continuous embedding of $H^1$ into $L^\infty$.
More general conditions on the potential $d(\cdot)$ have also been considered in \cite{FengProhl03}.
\end{remark}
The considerations of this section show that we can expect unique solutions and global well-posedness for the numerical approximations of the model problems \eqref{eq:ac} and \eqref{eq:hyb},
if the time step $\tau$ is chosen sufficiently small. This is also observed in our computations.

\section{Details on the implementation} \label{sec:implementation}

Before we turn to numerical tests, let us briefly discuss some details of the implementation that will be used in our computations later on. This particularly involves the choice of the approximation spaces $V_h$ and the solution of the nonlinear systems \eqref{eq:variationalhh}
in every time step.

\subsection{A finite element method}
Let $T_h=\{T\}$ denote a non-overlapping partition of the domain $\Omega \subset \RR^d$, $d=2,3$ 
into triangles or tetrahedra. We assume that the partition is conforming and uniformly shape-regular in the usual sense \cite{Ciarlet78,ErnGuermond04}. We then choose
\begin{align*}
V_h = \{v_h \in C(\Omega) : v|_T \in P_k(T)\} \subset H^1(\Omega)
\end{align*}
as the space of continuous piecewise polynomials of degree $k \ge 1$.
In our numerical tests we will actually only utilize the lowest order case $k=1$,
but the methods can in principle be extended to higher order without difficulty.
Similar approximations in space have been considered in \cite{ChenElliottGardinerZhao98,DuNicolaides91}.

\subsection{Nonlinear potential}

% As a second topic, let us briefly discuss the nonlinear term stemming from the double well potential. 
Without loss of generality, one may assume that the potential $d(\cdot)$ is a polynomial function. 
Otherwise, $d(\cdot)$ can be replaced by a polynomial and due to the Weierstrass theorem, the error of this approximation can 
be made arbitrarily small on bounded intervals.
In our numerical tests, we will utilize a function of the form 
\begin{align*}
d(S) = \psi(S) + F(S) = 4 S^2 (1-S)^2 + C S, \qquad C \in \RR. 
\end{align*}
The first part is the usual double-well potential and the linear term amounts to external forces. 
Note that the potential $d(\cdot)$ does not strictly satisfy the assumptions (A2), in particular, the upper bounds on the second derivative
are violated. 
As noted in Remark~\ref{rem:ass}, this assumption was made for convenience and is not required as long as $S$ stays uniformly bounded.
For the above choice of the potential $d(S)$, the difference quotient arising in \eqref{eq:variationalhh}
can be expressed as 
\begin{align*}
D(S,S^{n-1}) &= 
\frac{d(S) - d( S^{n-1})}{S - S^{n-1}} \\
&= 4 (S^3 + S^2 (S^{n-1} - 2) + S (S^{n-1}-1)^2 + S^{n-1} (S^{n-1}-1)^2 ) + C.   
\end{align*}
The nonlinear problem \eqref{eq:variationalhh} that has to be solved in every time step of the fully discrete scheme thus 
  has a polynomial nonlinearity and is non-degenerate. Standard iterative methods can therefore be utilized for an efficient solution.

\subsection{Iterative solver}

The following fixed-point iteration will be used in our numerical tests for the 
solution of the nonlinear problem that arises in every time step of the fully discrete scheme. 
\begin{problem}[Fixed-point iteration]  \label{prob:iter}
Let $S_h^{n-1} \in V_h$ and $\gamma \ge 0$, $\tau>0$ be given and define $c^{n-1} = c(|\nabla S_h^{n-1}|)$.
Further set $\tilde S_h^{n,0}=S_h^{n-1}$ and for $k \ge 1$, find $S_h^{n,k} \in V_h$ such that 
\begin{align*} % \label{eq:iter}
\left(\frac{c^{n-1}}{\tau} S_h^{n,k}, v_h\right) &+ \alpha (\nabla S_h^{n,k}, \nabla v_h) + \gamma (S_h^{n,k},v_h)\\
&= \left(\frac{c^{n-1}}{\tau} S_h^{n-1},v_h\right) - \beta (D(S_h^{n,k-1},S_h^{n-1}), v_h) + \gamma (S_h^{n,k-1},v_h)
\end{align*}
for all $v_h \in V_h$. If convergence is reached after iteration $k=k^*$, set $S_h^{n+1}=S_h^{n,k^*}$.
\end{problem}
Choosing $\gamma \approx \beta$ allows to improve the convergence behavior also in case $\beta \gg 0$;
see \cite{FengSongTangYang13} for a similar argument.
As an immediate consequence of the particular construction, we obtain 
\begin{lemma}
For any $\alpha,\beta,\gamma \ge 0$, the above fixed-point iteration is well-defined. 
If $\tau>0$ is  sufficiently small, then the iteration converges to the unique solution of Problem~\ref{eq:variationalhh}.
\end{lemma}
\begin{proof}
First note that the problem that has to be solved in every step of the fixed-point iteration of Problem~\ref{prob:iter} is linear. 
Moreover, the left hand side in the above iteration defines a positive definite quadratic form. 
This implies the existence of a unique solution $S_h^{n,k}$ in every step of the iteration. 
For $\tau>0$ sufficiently small, one can show that the map $\phi : S_h^{n,k-1} \to S_h^{n,k}$ is contractive, 
and convergence follows by the Banach fixed-point theorem.
\end{proof}

\begin{remark}
For time step $\tau>0$ sufficiently small, the contraction constant in the above iteration becomes small
and convergence will typically be observed within a few iterations. 
To speed up convergence, one may alternatively also utilize Newton-type iterations.   
\end{remark}

\section{Numerical illustration} \label{sec:numerics}

For our computations, we consider the Allen-Cahn equation \eqref{eq:ac} and the hybrid model \eqref{eq:hyb}. 
The model parameters are set to $\lambda=\mu=\nu=0.1$ in all tests, and the double-well potential is defined as $\psi(S) = 4 S^2 (1-S)^2$. As potential for the external forces, we choose $F(S)=C S$ with constant $C$ to be specified below. 
The constant $c$ in the Allen-Cahn equation is chosen as $c=\int_0^1 \sqrt{2 \psi(S)} dS \approx 0.4714$. 
For this setting, the interface velocities of Allen-Cahn and the hybrid model given in \eqref{eq:vac} and \eqref{eq:vhyb} agree up to higher order terms; see \cite{Alber15,AlberZhu13} for details.

For the space discretization, we utilize the finite element method described in Section~\ref{sec:implementation} with polynomial degree $k=1$.
The time integration is performed by the implicit time-stepping scheme \eqref{eq:variationalhh} with constant time step $\tau>0$, and the nonlinear systems arising in every time step are solved with the fixed-point iteration outlined above.

\subsection{A quasi one-dimensional problem}

As a first test case, we consider a quasi one-dimensional geometric setting. 
We choose $\Omega=(-3,3)^2$ and set 
\begin{align*}
S_0(x,y) = \chi_{\{|x| \le 3/2\}}(x,y),
\end{align*}
where $\chi_D$ is the characteristic function of the set $D$. 
The interface $\Gamma$ between the two phases at time $t=0$ 
is a straight line here and thue the curvature $\kappa_\Gamma$ is zero. 
This can be shown to remain true for all $t>0$ by a symmetry argument.
The evolution of the interface is therefore only driven by external forces in this example.
For our computations, we choose $F(S)=C S$ with $C=\frac{1}{2}$ 
which yields a constant driving force $F'(S)=\frac{1}{2}$ acting on the whole domain.

As mentioned above, the solution $S(x,y;t)$ can be shown to be independent of $y$ for all $t \ge 0$.
We thus obtain $S(x,y;t)=\widehat S(x;t)$ with $\widehat S$ defined by
\begin{alignat*}{5}
c(|\dx \widehat S|) \dt \widehat S &= \alpha \dxx \widehat S - \beta d(\widehat S), \qquad && x \in \omega, \ t>0, \\
\alpha \dx \widehat S &= 0, \qquad &&x \in \partial\omega, \ t>0.
\end{alignat*}
Here $\omega=(-3,3)$ is the one-dimensional cross-section at arbitrary $y$ and $\partial\omega=\{-3,3\}$. 
The initial condition for the one-dimensional problem is $\widehat S(x;0)=\chi_{\{|x|<3/2\}}(x)$. 
In Figure~\ref{fig:1d}, we display a few snapshots of the solution $S^{ac}(x,y;t)$ obtained with the Allen-Cahn equation \eqref{eq:ac}
in two dimensions and the corresponding solutions $\widehat S^{ac}$ and $\widehat S^{hyb}$ for the Allen-Cahn problem and the hybrid model in one space dimension. 
\begin{figure}[ht!] 
\centering
\includegraphics[width=0.2\textwidth]{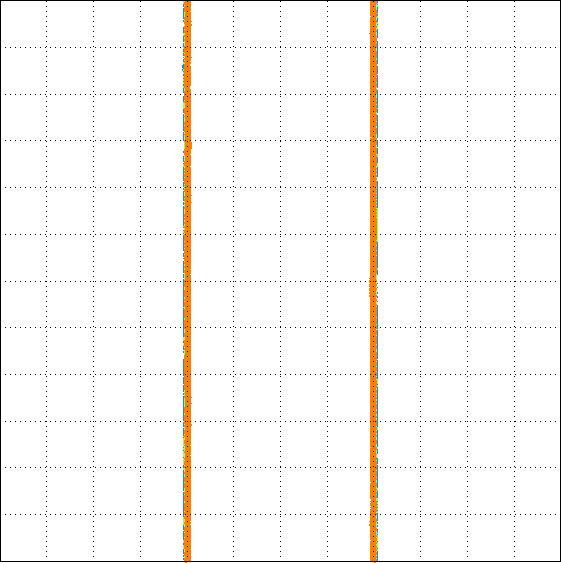} 
\hspace*{1ex}
\includegraphics[width=0.2\textwidth]{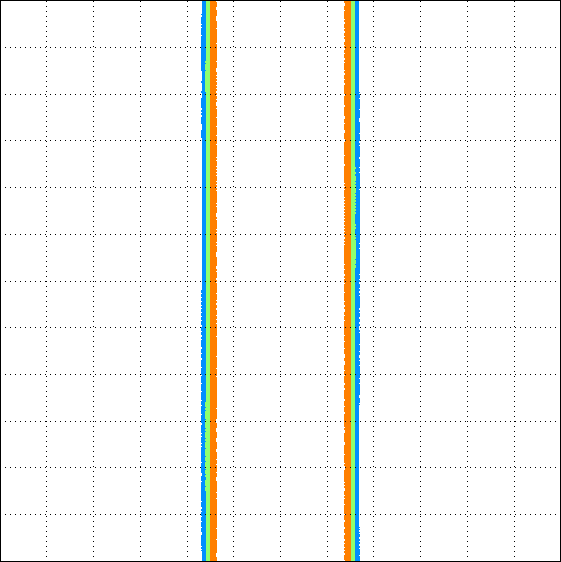} 
\hspace*{1ex}
\includegraphics[width=0.2\textwidth]{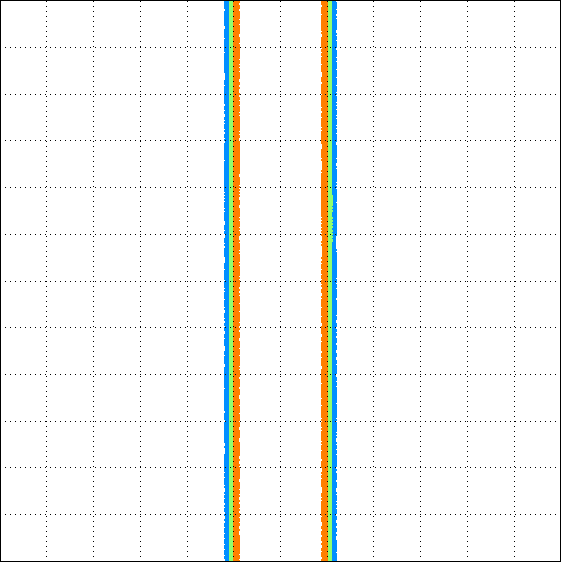} 
\hspace*{1ex}
\includegraphics[width=0.2\textwidth]{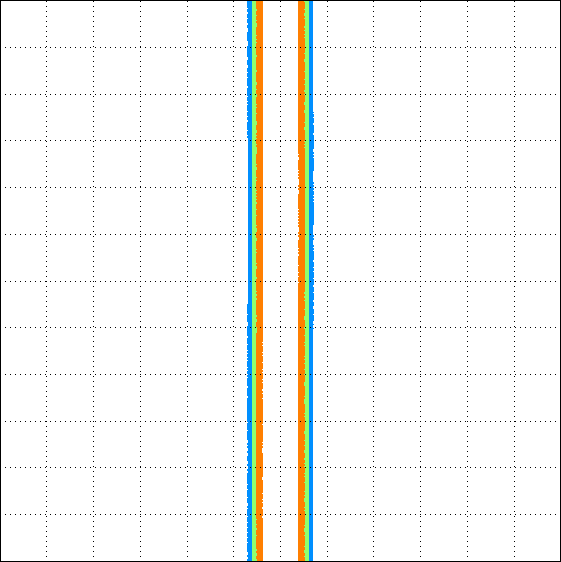} 
\\[3ex]
\includegraphics[width=0.2\textwidth]{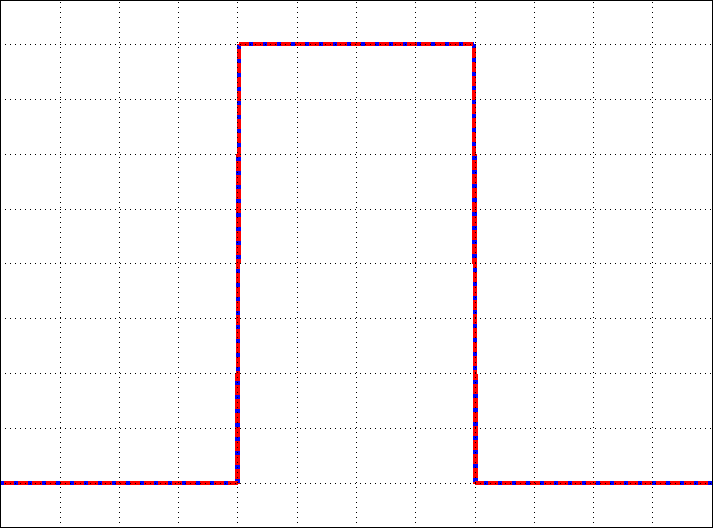} 
\hspace*{1ex}
\includegraphics[width=0.2\textwidth]{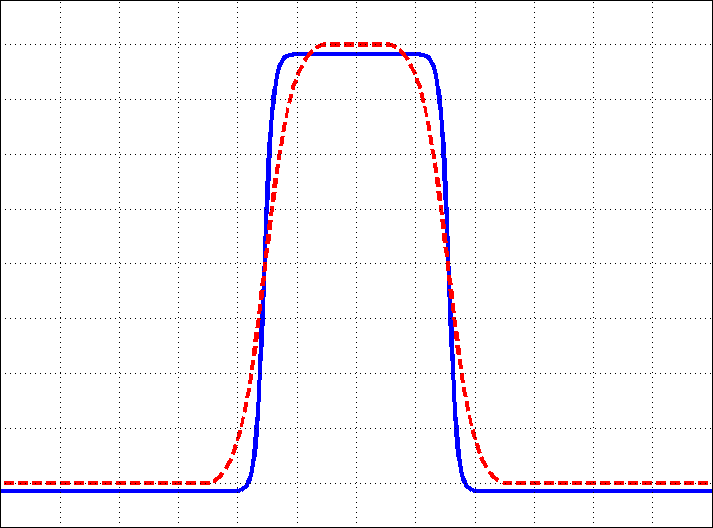} 
\hspace*{1ex}
\includegraphics[width=0.2\textwidth]{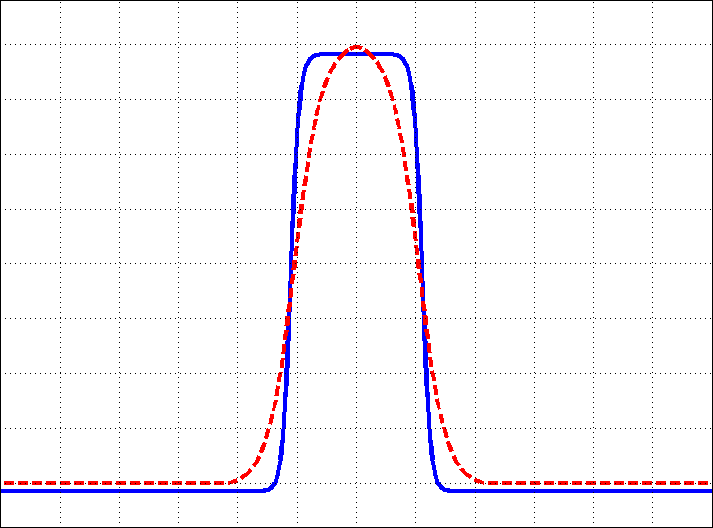} 
\hspace*{1ex}
\includegraphics[width=0.2\textwidth]{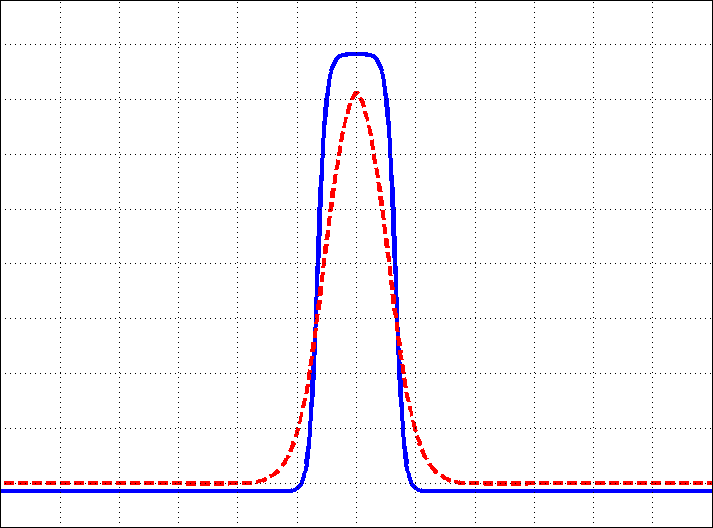} 
\caption{Interface $\Gamma(t)$ for the solution $S^{hyb}$ of the two-dimensional hybrid model at time $t=0,0.5,1,1.5$ (top) and solutions $\widehat S^{ac}$ and $\widehat S^{hyb}$ for the corresponding one-dimensional problems (bottom). 
The grid lines do not coincide with the mesh but are plotted only to allow for a better comparison.\label{fig:1d}} 
\end{figure}
The results clearly illustrate the relation $S^{ac}(x,y;t) = \widehat S^{ac}(x;t)$ between the solutions of the one- and the two-dimensional problem. 
Moreover, the propagation of interface $\widehat \Gamma=\{x : \widehat S(x)=1/2\}$ for the one-dimensional Allen-Cahn and the corresponding hybrid model take place at the same speed, 
which is in perfect agreement with formulas \eqref{eq:vac} and \eqref{eq:vhyb}.
In Figure~\ref{fig:area1d}, we display the area $A(S(t))=|\{x : \widehat S(x;t)>1/2\}|$ surrounded by the interface $\widehat \Gamma$ and we depict the decay of the energy functional for the two models illustrating the underlying gradient flow structure. 
\begin{figure}[ht!]
\centering
\includegraphics[width=0.7\textwidth]{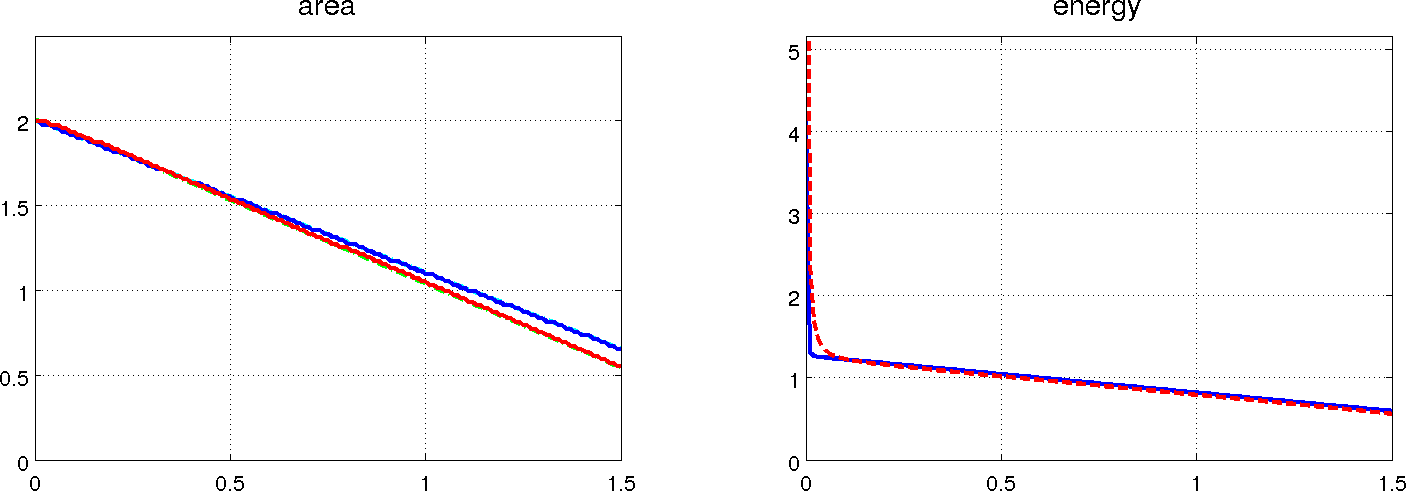} 
\\[2ex]
\caption{Left: Area $A(S(t))=\{x : \widehat S(x;t)>1/2\}$ for the one-dimensional Allen-Cahn model (blue) and the one- and two-dimensional hybrid model (red on top of green). Right: Energy $E(\widehat S(t)) = \int_\omega \frac{\alpha}{2} |\dx \widehat S(x;t)|^2 + \beta d(\widehat S(x;t)) \; dx$ for the one-dimensional Allen-Cahn (blue) and hybrid (red) model. \label{fig:area1d}} 
\end{figure}

\subsection{Shrinking of a circle by mean curvature}

As a second test case, we consider the evolution of a circular interface $\Gamma$. 
We again choose $\Omega=(-3,3)^2$ and define the initial values as
\begin{align*}
S_0(x,y) = \chi_{\{(x,y) : x^2+y^2<9/4\}}(x,y). 
\end{align*}
Here we set $F(S)=0$ which yields $[C]=0$ and thus the motion is only 
driven by the curvature of the interface. 
The solution of the sharp interface limit can be shown to be radially symmetric and 
according to \eqref{eq:vsi}, the propagation of the interface is governed by 
\begin{align*}
r'(t) = c \lambda^{1/2} \kappa_{\Gamma(t)} = -c\lambda^{1/2} \frac{1}{r(t)}, \qquad r(0)=3/2;
\end{align*}
here $r(t)$ denotes the radial position of the interface $\Gamma(t)$ at time $t$. 
This allows to compute the area $A(t)$ surrounded by the interface $\Gamma(t)$ analytically, i.e.,
\begin{align*}
A(t) = \pi r(t)^2 = \max\{\pi r(0)^2 - 2 \pi c \lambda^{1/2} t,0\}.
\end{align*}
In Figure~\ref{fig:circle2d}, we display a few snapshots of the interface $\Gamma(t)$ for the numerical solutions 
of the Allen-Cahn model \eqref{eq:ac} and the hybrid model \eqref{eq:hyb}. 
\begin{figure}[ht!]
\centering
\includegraphics[width=0.2\textwidth]{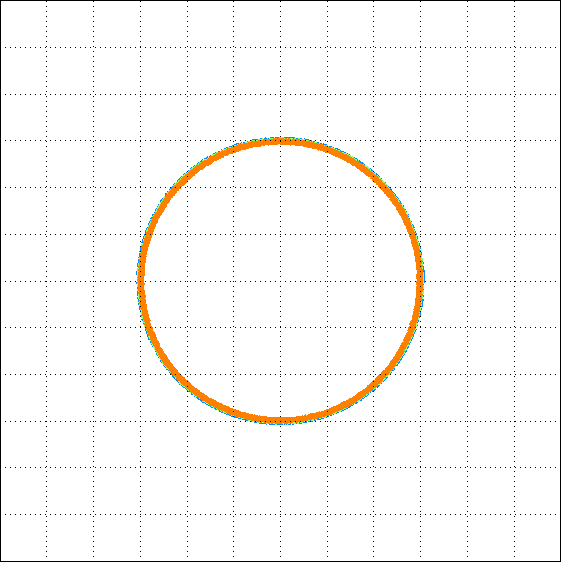} 
\hspace*{1ex}
\includegraphics[width=0.2\textwidth]{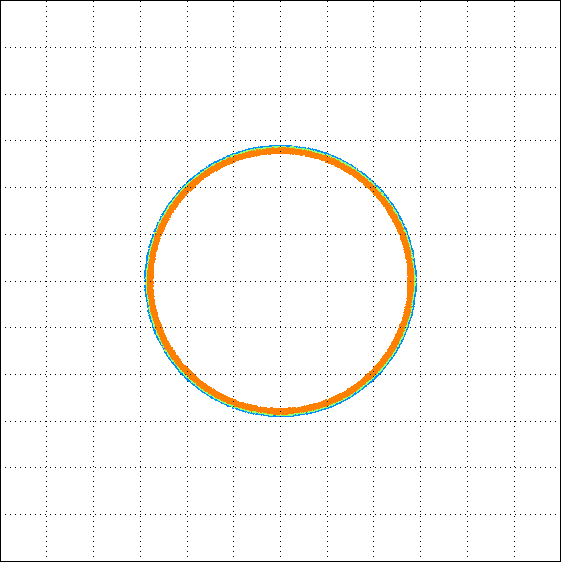} 
\hspace*{1ex}
\includegraphics[width=0.2\textwidth]{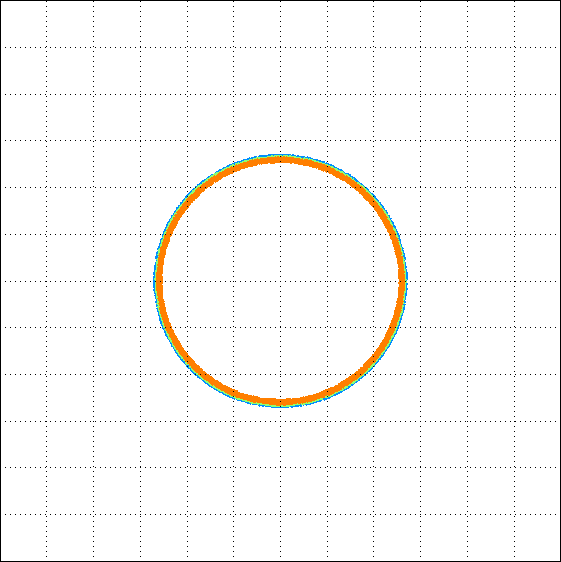} 
\hspace*{1ex}
\includegraphics[width=0.2\textwidth]{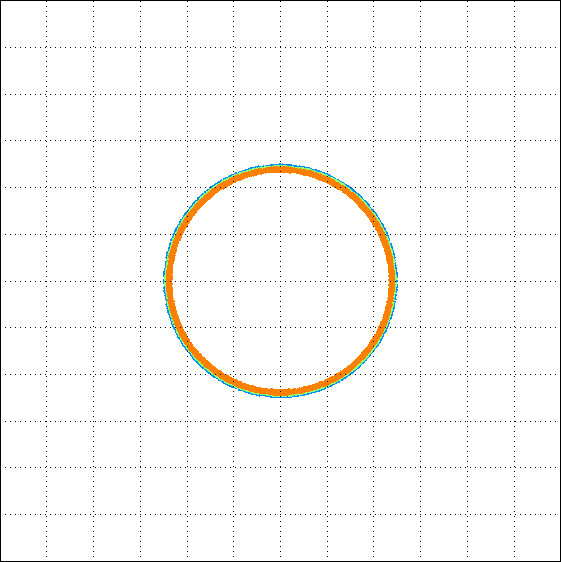} 
\\[3ex]
\includegraphics[width=0.2\textwidth]{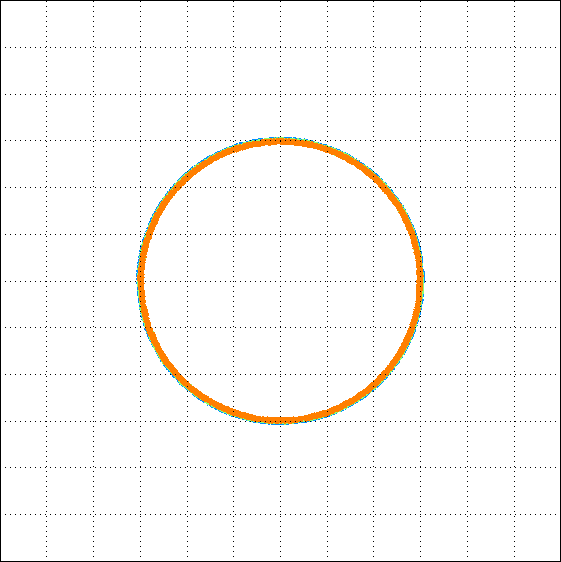} 
\hspace*{1ex}
\includegraphics[width=0.2\textwidth]{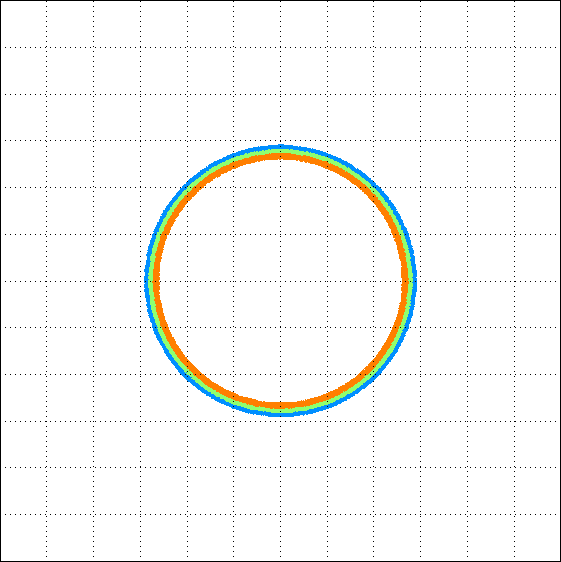} 
\hspace*{1ex}
\includegraphics[width=0.2\textwidth]{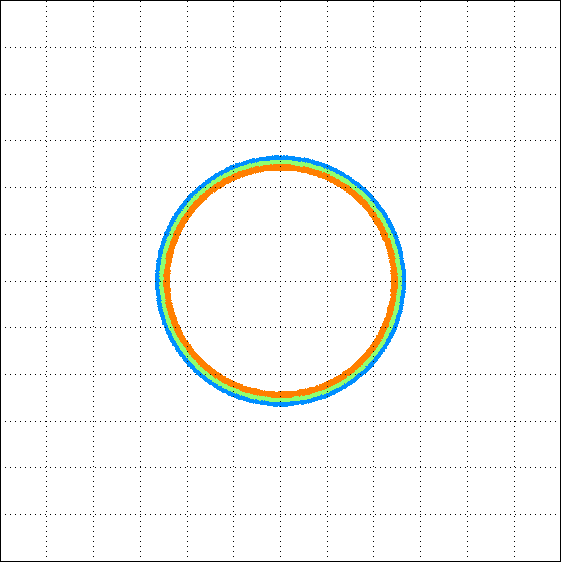} 
\hspace*{1ex}
\includegraphics[width=0.2\textwidth]{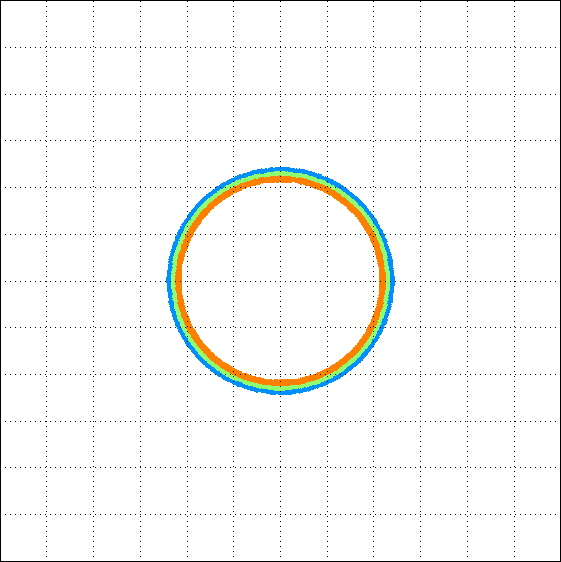} 
\\[2ex]
\caption{Interface $\Gamma(t)$ for solutions $S^{ac}$ and $S^{hyb}$ of the Allen-Cahn equation (top) and the corresponding hybrid model (bottom) for $t=0,1,2,3$.\label{fig:circle2d}.} 
\end{figure}
A brief visual inspection shows that the motion of the interface is rather similar for both models.
One can further deduce from the plots that the width of the transition zone in the hybrid model is somewhat larger than for the Allen-Cahn model; compare with our discussion in the introduction. 
A coarser mesh might therefore be used for the simulation of the interface evolution by the hybrid model. For a better comparison of the two models, we again display in Figure~\ref{fig:area_circle2d} the evolution of the area $A(t)$ and of the energy $E(S(t))$ representing one phase of the system.
\begin{figure}[ht!]
\centering
\includegraphics[width=0.7\textwidth]{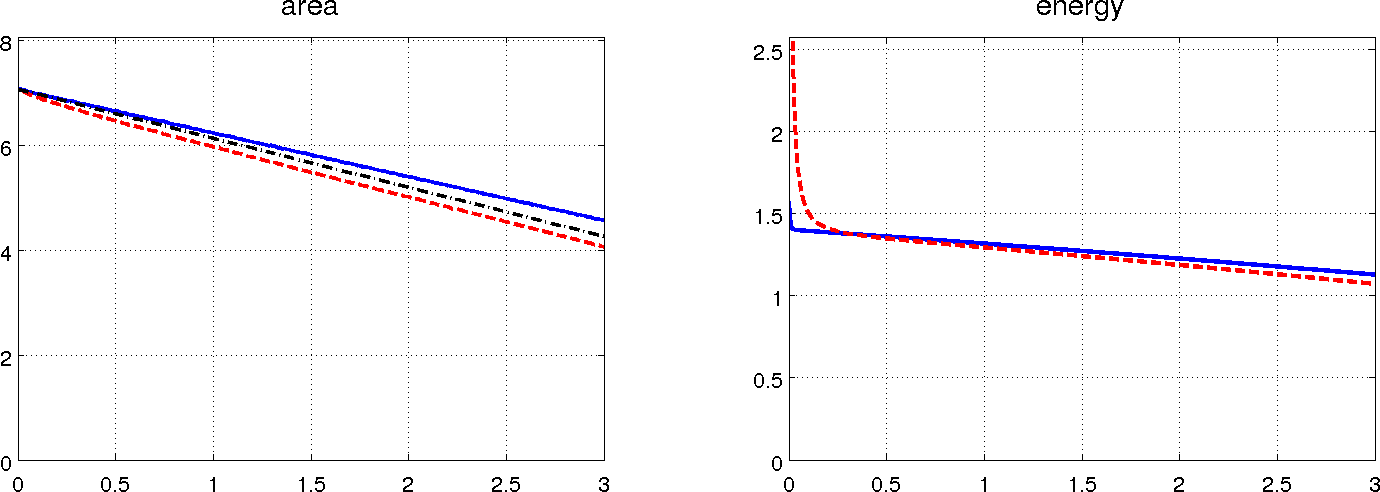} 
\\[2ex]
\caption{Left: area $A(t)$ enclosed by the interface $\Gamma(t)$ for the Allen-Cahn model (blue), the hybrid model (red),
and the exact solution of the sharp inferface limit (black). 
Right: Evolution of the energy $E(S(t))$.\label{fig:area_circle2d}} 
\end{figure}
In accordance with formulas \eqref{eq:vac} and \eqref{eq:vhyb} for our choice of parameters, 
the evolution of the interface here occurs at approximately the same speed. 
Moreover, the energy $E(S(t))$ is monotonically decreasing for both models which is in perfect agreement 
with the assertions of Lemma~\ref{lem:decayhh}. 

\subsection{A non-smooth and non-convex geometry}

As a last test case, we now consider the evolution of an interface which initially is non-smooth and encloses a non-convex region at time $t=0$.
Some snapshots of the evolution are depicted in Figure~\ref{fig:general2d}.
\begin{figure}[ht!]
\centering
\includegraphics[width=0.2\textwidth]{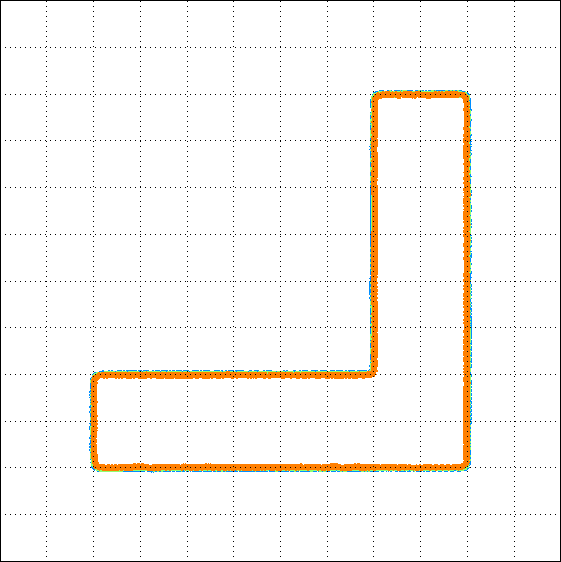} 
\hspace*{1ex}
\includegraphics[width=0.2\textwidth]{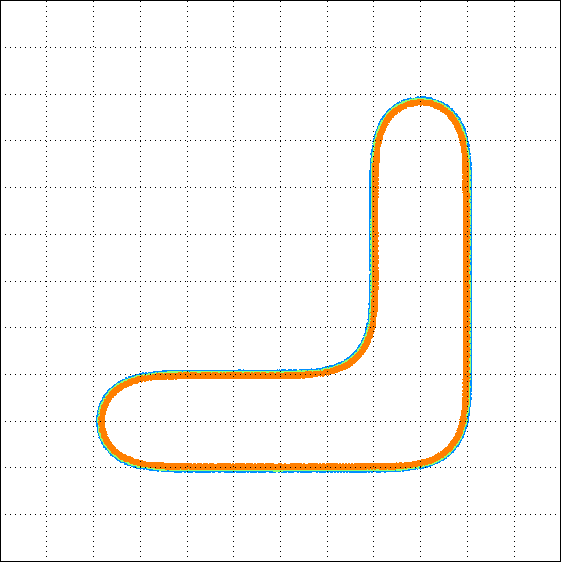} 
\hspace*{1ex}
\includegraphics[width=0.2\textwidth]{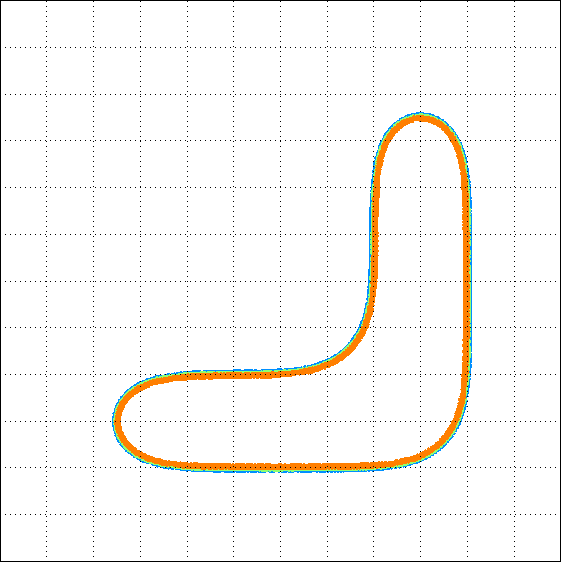} 
\hspace*{1ex}
\includegraphics[width=0.2\textwidth]{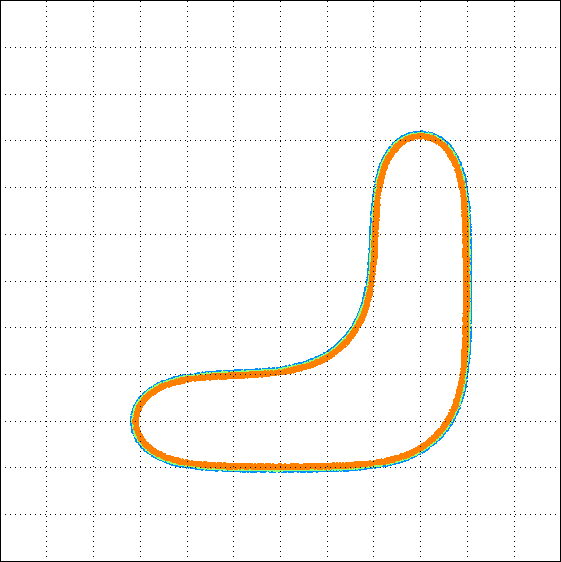} 
\\[3ex]
\includegraphics[width=0.2\textwidth]{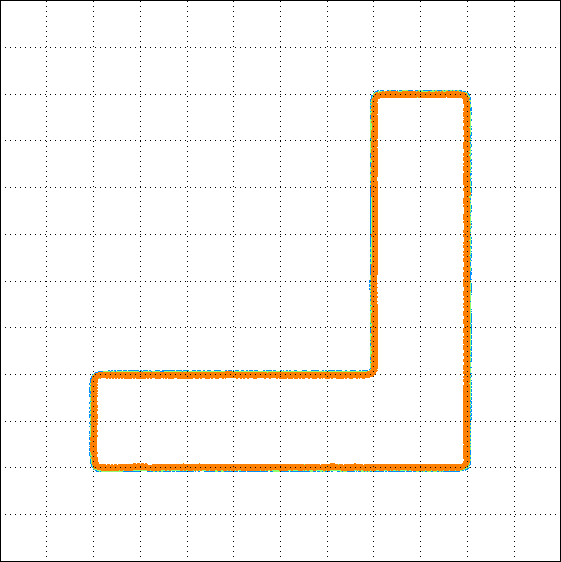} 
\hspace*{1ex}
\includegraphics[width=0.2\textwidth]{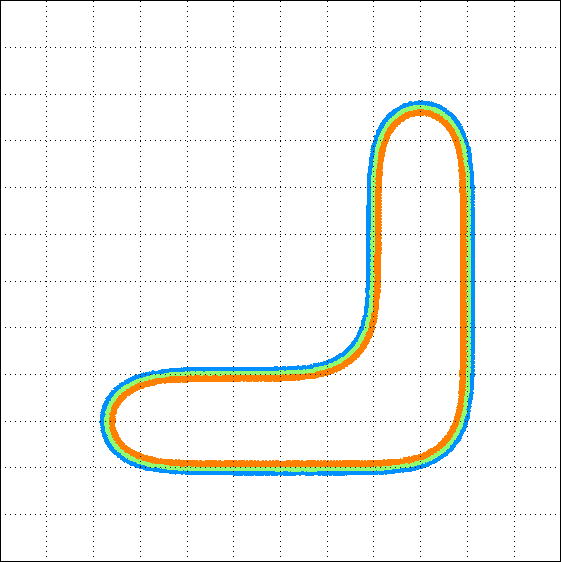} 
\hspace*{1ex}
\includegraphics[width=0.2\textwidth]{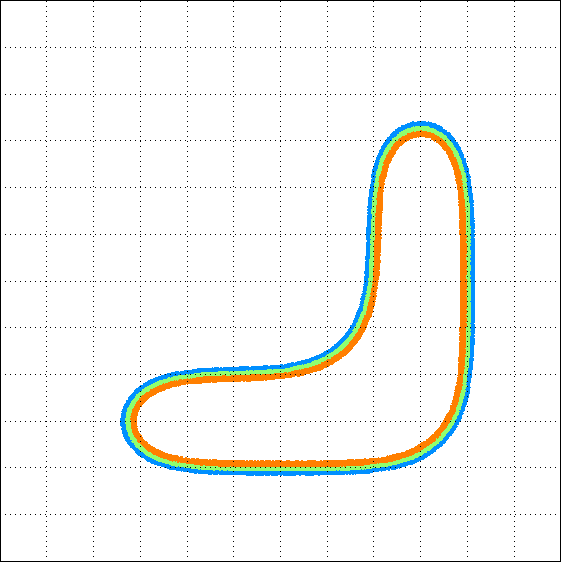} 
\hspace*{1ex}
\includegraphics[width=0.2\textwidth]{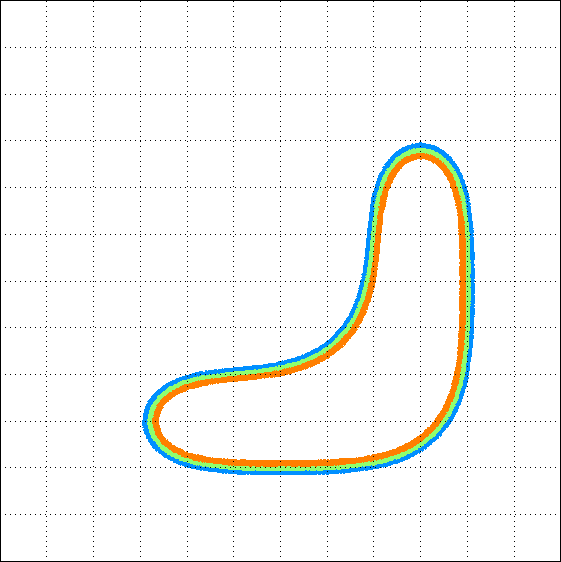} 
\\[2ex]
\caption{Evolution of the interface for Allen-Cahn model (left) and hybrid model (right) for $t=0,0.5,1,1.5$ (top to bottom)\label{fig:general2d}} 
\end{figure}
Like in the previous test, we set the force potental to $F(S)=0$ which implies $[C]=0$, 
and hence the motion is again driven only by the curvature. 
Note that in the vicinity of convex corners, the interface moves towards the interior of the enclosed area
while at the non-convex corners, the interface propagates in the other direction, which is caused by the change of sign 
in the curvature.
In Figure~\ref{fig:area_general2d}, we again depict the area surrounded by the interface and the uniform decay of the energy for the numerical solutions obtained with the two models. 
\begin{figure}[ht!]
\centering
\includegraphics[width=0.7\textwidth]{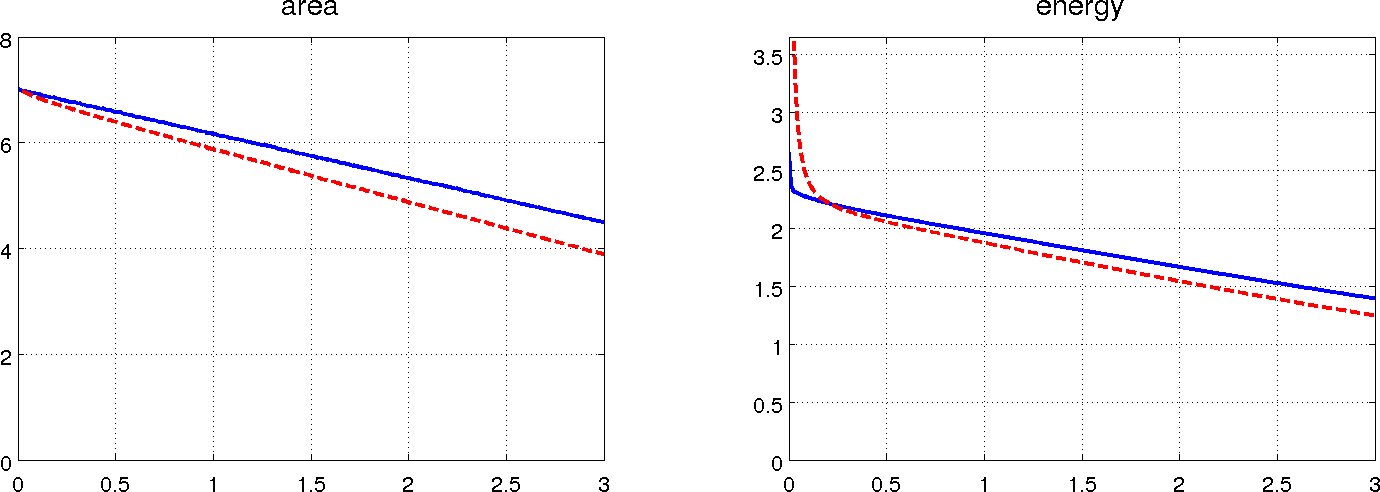} 
\\[2ex]
\caption{Left: area $A(t)$ enclosed by the interface $\Gamma(t)$ for Allen-Cahn (blue) and the hybrid model (red). 
Right: Evolution of the energies $E(S(t))$.\label{fig:area_general2d}} 
\end{figure}
Like in the previous examples, the velocity of the interface motion in the Allen-Cahn and the hybrid model are very similar; compare with the formulas \eqref{eq:vac} and \eqref{eq:vhyb}. Due to the gradient flow structure of the fully discrete evolution, a strict decrease in the energy $E(S(t))$ can again be observed.  
The results are very similar, and the slight discrepancies can be explained by approximation errors.

\section{Discussion} \label{sec:discussion}

In this paper, we investigated the systematic numerical approximation of a general class of Allen-Cahn type equations.
The common feature of these models was a gradient flow structure with respect to an associated energy. 
We proposed and analyzed semi-discrete and fully discrete numerical schemes which strictly preserve the underlying gradient flow structure and which therefore automatically yield uniformly energy stable discrete approximations. 
Well-posedness of the numerical schemes and energy decay was established theoretically and illustrated by numerical tests.
Our computational results also provide further numerical evidence for asymptotic expansions of interface motion obtained \cite{Alber15,AlberZhu13}.

In our analysis, we used some rather strong structural assumptions on the coefficients arising in the equations, which should allow to conduct a full a-priori convergence analysis of the semi- and fully discrete scheme by following the arguments of \cite{ChenElliottGardinerZhao98,FengProhl03}. 
As indicated in remarks, a detailed inspection of the proofs allows to further relax these assumptions. 
Our arguments and methods of proof are rather general and seem to be applicable also to phase-field models for the dynamics of solid-solid phase transition which are governed by coupled Allen-Cahn and elasticity equations.
These topics and further extensions are left for future research.

\section*{Acknowledgments}
This work was supported by the German Research Foundation (DFG) via grants IRTG~1529, TRR~154, and Eg-331/1-1
and by the German Excellence Initiative via grant GSC~233.

% \bibliographystyle{abbrv}
% \bibliography{allencahn}

\end{document}